\documentclass[12pt]{article}

\usepackage{amsmath, amsthm, amssymb}

\usepackage{fullpage}

\usepackage{tikz}
\usetikzlibrary{patterns}
\usetikzlibrary{shapes}
\usepackage{fullpage}
\usepackage{hyperref}
\usepackage{ytableau}
\usepackage{caption}
\usepackage{subcaption}
\usepackage{epsfig,multicol}

\newtheorem{theorem}{Theorem}
\newtheorem{lemma}{Lemma}

\newtheorem{definition}{Definition}

\title{$s$-Catalan numbers and Littlewood-Richardson polynomials}
\author{William Linz\thanks{Department of Mathematical Sciences, University of Illinois at Urbana-Champaign, IL, USA. Email: \texttt{wlinz2@illinois.edu}. Partially supported by NSF RTG Grant DMS-1937241.}}
\date{\today}

\begin{document}
\maketitle

\begin{abstract}
In this note, we study two generalizations of the Catalan numbers, namely the $s$-Catalan numbers and the spin $s$-Catalan numbers. These numbers first appeared in relation to quantum physics problems about spin multiplicities. We give a combinatorial description for these numbers in terms of Littlewood-Richardson coefficients, and explain some of the properties they exhibit in terms of Littlewood-Richardson polynomials.
\end{abstract}

The purpose of this note is to study two generalizations of the well-known Catalan numbers and prove some basic combinatorial properties of these numbers by characterizing them as Littlewood-Richardson coefficients of certain triples of partitions. These numbers first appeared in quantum physics problems about spin multiplicities~\cite{CHI, CKZ}, and have also been studied in special cases by Belbachir and Iguerofa ~\cite{BI2}. 

We begin with the \emph{$s$-binomial coefficients}, a generalization of the ordinary binomial coefficients. 

\begin{definition}[$s$-binomial coefficients]\label{defn:sbin}
For positive integers $n$ and $s$, and integer $k$ with $0\le k\le sn$, the $s$-binomial coefficient $\binom{n}{k}_s$ is the coefficient of $x^k$ in the expansion of $(1+x+x^2+\cdots +x^s)^n$:
\[(1+x+x^2+\cdots +x^s)^n = \sum_{k=0}^{sn}\binom{n}{k}_s x^{k}.\]
\end{definition}

The $s$-binomial coefficients were apparently first defined by De Moivre~\cite{DM}, and have since been studied by a number of mathematicians, including Euler \cite{E} in special cases as part of his investigations on generating series, which are now a basic tool in combinatorics.  Belbachir and Igueroufa~\cite{BI1} give an overview of the history of the $s$-binomial coefficients\footnote{Belbachir and Iguerofa refer to the $s$-binomial coefficients as \emph{bi$^{s}$nomial coefficients}.} and their basic properties, including expansions of the $s$-binomial coefficients in terms of the ordinary binomial coefficients. 

Andr\'e~\cite{A} was apparently the first to study $s$-binomial coefficients combinatorially for general $s$. Freund~\cite{F} independently considered $s$-binomial coefficients combinatorially, in terms of restricted occupancy problems, and we recall his combinatorial description (which is basically equivalent to the definition given in Definition~\ref{defn:sbin}). 

\begin{lemma}[Freund]\label{lem:fdlemma}
The $s$-binomial coefficients $\binom{n}{k}_s$ count the number of ways of placing $k$ objects in $n$ boxes, provided each box contains at most $s$ objects. 
\end{lemma}

In analogy with the ordinary binomial coefficients, we define the \emph{central $s$-binomial coefficients} to be the coefficients $\binom{2n}{sn}_s$. We can similarly define the \emph{$s$-Catalan numbers}. 
\begin{definition}[$s$-Catalan numbers]\label{defn:scat}
Let $n$ and $s$ be positive integers. The $n$th $s$-Catalan number is the number \[\binom{2n}{sn}_s - \binom{2n}{sn+1}_s.\]
\end{definition}
We give a list of small $s$-Catalan numbers for $s=1, 2, \ldots, 7$ in Figure~\ref{fig:scattable}. 

\begin{figure}
\centering
\begin{tabular}{c|c|c|c|c|c|c|c|}
 & $s=1$ & $s=2$ & $s=3$ & $s=4$ & $s=5$ & $s=6$ & $s=7$ \\
 \hline
 $n = 1$ & 1 & 1 & 1 & 1 & 1 & 1 & 1 \\
 $n = 2$ &  2 & 3 & 4 & 5 & 6 & 7 & 8 \\
 $n = 3$ & 5 & 15 & 34 & 65 & 111 & 175 & 260 \\
 $n = 4$ & 14 & 91 & 364 & 1085 & 2666 & 5719 & 11096 \\
 $n = 5$ & 42 & 603 & 4269 & 19845 & 70146 & 204867 & 518498 \\
 $n = 6$ & 132 & 4213 & 52844 & 383251 & 1949156 & 7737807 & 25593128\\
 $n = 7$& 429 & 30537 & 679172 & 7687615 & 56267133 & 303922983 & 1312660700\\
 $n = 8$ & 1430 & 227475 & 8796188 & 158614405 & 1670963202 & 12281450455 & 69270071480\\
\end{tabular}
\caption{A list of $s$-Catalan numbers for $1\le s\le 7$ and $1\le n\le 8$.}
\label{fig:scattable}
\end{figure}

Belbachir and Iguerofa~\cite{BI2} defined the $s$-Catalan numbers only for odd positive integers $s$, but the definition makes sense for all positive integers $s$. When $s=1$, the $s$-Catalan numbers reduce to the ordinary Catalan numbers, for which there are many known combinatorial interpretations~\cite{Stan}. Belbachir and Iguerofa~\cite{BI2} ask for a combinatorial interpretation of the $s$-Catalan numbers. 

As we mentioned in the opening, motivation for studying the $s$-Catalan numbers comes from certain problems in physics. Specifically, Cohen, Hansen and Itzhaki~\cite{CHI} and Curtright, van Kortryk, and Zachos~\cite{CKZ} were interested in spin multiplicities that occur in spin $s$ $n$-fold tensor products. That is, given a Lie group $G$ (in \cite{CHI}, this Lie group is $SO(3)$; in \cite{CKZ}, it is $SU(2)$), one is interested in determining the multiplicity of the irreducible representation $[j]$ in the tensor product $[s]^{\otimes n}$, where $s$ can be any nonnegative half-integer (\textit{i.e.}, $s=t/2$ for some nonnegative integer $t$). In the case that $j=0$, the multiplicities calculated by Cohen, Hansen and Itzhaki are given by a generalization of the $s$-Catalan numbers, which we call the \emph{spin $s$-Catalan numbers}.

\begin{definition}[spin $s$-Catalan numbers]\label{s2catnum}
Let $n$ be a positive integer, and let $s$ be a half-integer (\textit{i.e.}, a positive integer divided by $2$). The $n$th spin $s$-Catalan number is the number
\[\binom{n}{sn+s}_{2s} - \binom{n}{sn+s+1}_{2s}.\]
If $n$ is even and $s$ is not an integer, then we consider the $n$th spin $s$-Catalan number to be $0$.
\end{definition}

A table of small spin $s$-Catalan numbers is given in Figure~\ref{fig:s2cattable}.

Curtright, van Kortryk, and Zachos produced a table similar to the one in Figure~\ref{fig:s2cattable}, and noted that the rows of the table are polynomials. Indeed, if we ignore the 0 entries in Figure~\ref{fig:s2cattable}, then the rows are polynomials and we will show that these rows are certain Littlewood-Richardson polynomials. 

We now give representations of the $s$-Catalan numbers and spin $s$-Catalan numbers as certain Littlewood-Richardson coefficients. As the Littlewood-Richardson coefficients have a well-known combinatorial interpretation, which we recall shortly, this also provides one answer to the question of Belbachir and Iguerofa. Let us recall the definition of the Littlewood-Richardson coefficients (see, for example, Barcelo and Ram \cite{BR}). 

\begin{figure}
\centering
\begin{tabular}{c|c|c|c|c|c|c|c|}
 & $s=\frac12$ & $s=1$ & $s=\frac32$ & $s=2$ & $s=\frac52$ & $s=3$ & $s=\frac72$ \\
 \hline
 $n = 1$ & 1 & 1 & 1 & 1 & 1 & 1 & 1 \\
 $n = 2$ &  0 & 1 & 0 & 1 & 0 & 1 & 0 \\
 $n = 3$ & 2 & 3 & 4 & 5 & 6 & 7 & 8 \\
 $n = 4$ & 0 & 6 & 0 & 16 & 0 & 31 & 0 \\
 $n = 5$ & 5 & 15 & 34 & 65 & 111 & 175 & 260 \\
 $n = 6$ & 0 & 36 & 0 & 260 & 0 & 981 & 0\\
 $n = 7$& 14 & 91 & 364 & 1085 & 2666 & 5719 & 11096\\
 $n = 8$ & 0 & 232 & 0 & 4600 & 0 & 33922 & 0\\
 $n = 9$ & 42 & 603 & 4269 & 19845 & 70146 & 204867 & 518498\\
 $n = 10$ & 0 & 1585  & 0 &  86725 & 0 &1251460 &0 \\
\end{tabular}
\caption{A list of spin $s$-Catalan numbers for $\frac12\le s\le \frac72$ and $1\le n\le 10$.}
\label{fig:s2cattable}
\end{figure}

\begin{definition}[Littlewood-Richardson coefficients]
Let $V^{\mu}$ and $V^{\nu}$ be irreducible polynomial representations of the group $GL(n)$. The Littlewood-Richardson coefficients $c^{\lambda}_{\mu, \nu}$ are the coefficients
\[V^{\mu}\otimes V^{\nu} = \sum_{\lambda}c^{\lambda}_{\mu, \nu} V^{\lambda}.\]
\end{definition}

As is well-known, the Littlewood-Richardson coefficients amazingly have a combinatorial interpretation in terms of Littlewood-Richardson tableaux (see, for instance, Fomin's appendex to Stanley's book~\cite{Fo} or the book of Manivel~\cite{M}). 

Recall that for partitions $\lambda$ and $\mu$, the skew shape $\lambda / \mu$ is the set-theoretic difference of the Young diagrams of $\lambda$ and $\mu$. If $\nu = (\nu_1, \nu_2, \ldots)$, then the shape is defined to have content $\nu$ if, in a filling $F$ of $\lambda / \mu$, there are exactly $\nu_i$ $i$s in $F$ for each $i$. A tableau on the shape $\lambda / \mu$ is semistandard if the associated filling is weakly increasing along rows and strictly increasing along columns. Finally, a word $a_1a_2\cdots a_n$ is a ballot word if, for each initial segment $a_1a_2\cdots a_j$, the number of $a$s appearing in the word is at least the number of $b$s appearing in the word whenever $a < b$. 

\begin{theorem}[Littlewood-Richardson coefficients (Combinatorial interpretation)]
For a triple of partitions $(\lambda, \mu, \nu)$, the Littlewood-Richardson coefficient $c^{\lambda}_{\mu, \nu}$ is equal to the number of semistandard skew tableaux of shape $\lambda / \mu $ and content $\nu$, and for which the associated word obtained by reading the reversed rows concatentated together is a ballot word (or Yamanouchi word, or lattice permutation). 
\end{theorem}

For example, we have $c^{(3, 2, 1)}_{(1), (3, 2)} = 1$, because  the only valid Littlewood-Richardson tableau is
\[
\ytableausetup{notabloids}
\begin{ytableau}
\none & 1& 1\\
1&2\\
2
\end{ytableau}
\] 

Note that the row-reversed word of the above tableau is $11212$. It is this combinatorial interpretation of the Littlewood-Richardson coefficients that we subsequently use in our proofs.

For a partition $\lambda$, we will use $s*\lambda$ to denote the partition obtained from $\lambda$ by a stretching factor of $s$. For example, if $\lambda = (3, 2, 1)$ and $s=3$, then $s*\lambda = (9, 6, 3)$.

The Littlewood-Richardson coefficients exhibit a remarkable stretching property: for any triple of partitions, $\lambda, \mu, \nu$, the coefficients $c^{N*\lambda}_{N*\mu, N*\nu}$ are polynomials in the stretching factor $N$; that is, there is a polynomial $P^{\lambda}_{\mu, \nu}$ with rational coefficients such that 
\[c^{N*\lambda}_{N*\mu, N*\nu} = P^{\lambda}_{\mu, \nu}(N).\]

We will refer to any such polynomial as a \emph{Littlewood-Richardson polynomial}. The polynomiality of these stretched Littlewood-Richardson coefficients was first conjectured by King, Tollu and Toumazet~\cite{KTT}, and subsequently established independently by Derksen and Weyman~\cite{DW} and Rassert~\cite{R}. King, Tollu and Toumazet also conjectured that the coefficients of Littlewood-Richardson polynomials $P^{\lambda}_{\mu, \nu}$ are nonnegative, but this remains open. 

We can represent both the $s$-Catalan and spin $s$-Catalan numbers as Littlewood-Richardson polynomials (in $s$), providing an explanation of the polynomiality of the rows in Figure \ref{fig:scattable} and Figure \ref{fig:s2cattable}. To state these theorems, for every positive integer $n$ define the partitions $\lambda_n = (n, n-1, \ldots, 1)$ and $\alpha_n = (2n, 2n-2, \ldots, 2) = 2*\lambda_n$. 

\begin{theorem}\label{thm:scatlrthm}
Let $s$ and $n$ be positive integers. Then, \[c_{s*\lambda_{2n-1}, s*(n, n)}^{s*\lambda_{2n}} = \binom{2n}{sn}_s - \binom{2n}{sn+1}_s.\]
\end{theorem}

\begin{theorem}\label{thm:laurentthm}
If $n$ and $s$ are positive integers, or if $s=k/2$ for some odd positive integer $k$ and $n$ is an odd positive integer, 
\[c^{s*\alpha_n}_{s*\alpha_{n-1}, s*(n+1, n-1)} = \binom{n}{sn+s}_{2s} - \binom{n}{sn+s+1}_{2s}.\]
\end{theorem}

The proofs of Theorems \ref{thm:scatlrthm} and \ref{thm:laurentthm} both follow by counting the number of Littlewood-Richardson tableaux of the appropriate shape and content, and using Lemma~\ref{lem:fdlemma} to give a bijection to the $s$-Catalan numbers and spin $s$-Catalan numbers. As the reflection principle argument is essentially the same for both, we omit the details for Theorem \ref{thm:laurentthm}, and only give the proof for Theorem \ref{thm:scatlrthm}. 

\begin{proof}[Proof of Theorem \ref{thm:scatlrthm}]
In the skew shape $s*\lambda_{2n} / s*{\lambda_{2n-1}}$, the condition that the filling must be strictly increasing on columns is automatically satisfied, because there are no two cells in the shape which are in the same column. Now it is straightforward to see from Lemma \ref{lem:fdlemma} that the number of possible fillings is $\binom{2n}{sn}_s$, by looking at the placement of the 1s in any such filling. 

We now show that the number of such fillings of $sn$ 1s and $sn$ 2s into the shape $s*\lambda_{2n} / s*{\lambda_{2n-1}}$ which do not satisfy the ballot condition is equal to $\binom{2n}{sn+1}_s$. We do this by exhibiting a bijection to the number of valid fillings (\textit{i.e.} those that weakly increase along rows) of $sn+1$ 1s and $sn-1$ 2s into the shape $s*\lambda_{2n} / s*{\lambda_{2n-1}}$. 

Let $F$ be a filling whose word does not satisfy the ballot condition. There must be a first position $k$ where the number of 2s read in the word exceeds the number of 1s. We form a filling $F'$ with $sn+1$ 1s as follows: replace each 1 in position $1$ through $k$ (inclusive) by a 2, and replace each $2$ by a $1$, adjusting the rows so that the $1$s weakly precede the $2$s. We can reverse this process as follows: Given a filling $F'$ of $sn+1$ 1s into the shape $s*\lambda_{2n} / s*{\lambda_{2n-1}}$, begin reading the word as follows: within each row, read the order of the word in reverse order, that is, reading the 1s before the 2s. Stop reading at the first position $j$ where the 1s exceed the 2s. As before, replace every $1$ in positions $1$ through $j$ with a $2$, and every $2$ with a $1$, adjusting rows as needed to fulfill the weakly increasing condition. 
\end{proof}

\begin{figure}[ht]
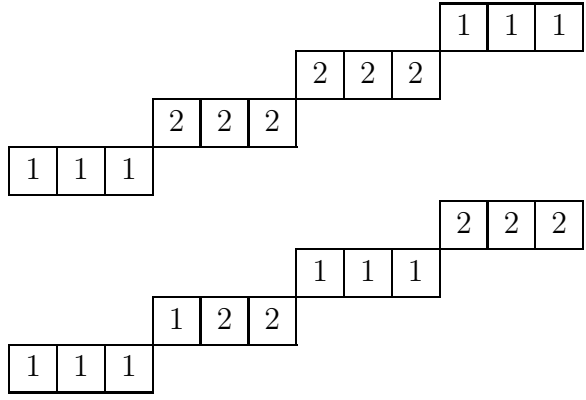

\centering 
\begin{subfigure}[b]{.4\textwidth}
\centering 
\ytableausetup{notabloids}
\begin{ytableau}
\none & \none & \none & \none & \none & \none & \none & \none & \none & 1 & 1 & 1\\
\none & \none & \none &\none &\none &\none & 2 & 2 & 2\\
\none & \none & \none & 2 & 2 & 2\\
1 & 1 & 1
\end{ytableau}
\label{s1}
\end{subfigure}
\centering
%{\LARGE$\leftrightarrow$}

\begin{subfigure}[b]{.4\textwidth}
\centering 
\ytableausetup{notabloids}
\begin{ytableau}
\none & \none & \none & \none & \none & \none & \none & \none & \none & 2 & 2 & 2\\
\none & \none & \none &\none &\none &\none & 1 & 1 & 1\\
\none & \none & \none & 1 & 2 & 2\\
1 & 1 & 1
\end{ytableau}
\label{s2}
\end{subfigure}

\caption{An example of the bijection in the proof of Theorem \ref{thm:scatlrthm} with $s=3, n=2$.}
\label{ex:s-Cat}
\end{figure}

Let us consider the relation between the $s$-Catalan and spin $s$-Catalan numbers a little more closely. It is not a coincidence that the odd rows of Figure \ref{fig:s2cattable} match up exactly with the rows of Figure \ref{fig:scattable}. To fix notation, set $n=2m+1$, with $m\ge 0$, and also set $s=t/2$, $t$ a positive integer (this $s$ corresponds to the $s$ for spin $s$-Catalan numbers). Then, $s*\alpha_n = \frac{t}{2}*\alpha_{2m+1} = t*\lambda_{2m+1}$, while the spin $(t/2)$-Catalan number at $n=2m+1$ is $\binom{2m+1}{t(m+1)}_t - \binom{2m+1}{t(m+1) + 1}_t$. Under these notational conventions, the $(2m+1)$th spin $(t/2)$-Catalan number corresponds to the $(m+1)$th $t$-Catalan number, which is the content of the next theorem. 

\begin{theorem}\label{thm:scatspinscatthm}
Let $t$ and $m$ be positive integers. Then, 
\[\binom{2m+1}{t(m+1)}_{t} - \binom{2m+1}{t(m+1)+1}_t = \binom{2m+2}{t(m+1)}_t - \binom{2m+2}{t(m+1)+1}_t,\]
or equivalently,
\[c^{t*\lambda_{2m+1}}_{t*\lambda_{2m}, t*(m+1, m)} = c^{t*\lambda_{2m+2}}_{t*\lambda_{2m+1}, t*(m+1, m+1)}.\]
\end{theorem}

Note that if $t=1$ in Theorem \ref{thm:scatspinscatthm}, then the binomial coefficient identity reduces to 
$\binom{2m+1}{m+1} - \binom{2m+1}{m+2} = \binom{2m+2}{m+1} -\binom{2m+2}{m+2}$, and the latter is clearly the $(m+1)$th Catalan number. 

\begin{proof}[Proof of Theorem \ref{thm:scatspinscatthm}]
We show that $c^{t*\lambda_{2m+2}}_{t*\lambda_{2m+1}, t*(m+1, m+1)} = c^{t*\lambda_{2m+1}}_{t*\lambda_{2m}, t*(m+1, m)}$. Any valid filling of a Littlewood-Richardson tableau on the skew shape $t*\lambda_{2m+2} / t*\lambda_{2m+1}$ with $t(m+1)$ 1s and $t(m+1)$ 2s must have a row of $t$ 2s in the bottom row, as any 1s in the bottom row would violate the ballot condition. Deleting the bottom row and the columns above it therefore leaves a valid Littlewood-Richardson tableau on the skew shape $t*\lambda_{2m+1} / t*\lambda_{2m}$ with $t(m+1)$ 1s and $tm$ 2s. This operation is clearly reversible, completing the proof. 
\end{proof}

\section*{Acknowledgements}
I thank Alexander Yong for stimulating my interest in these topics, in particular by asking for a description of the Littlewood-Richardson coefficients $c_{s*\lambda_{2n-1}, s*(n, n)}^{s*\lambda_{2n}}$, and for numerous helpful discussions and a number of suggestions which have improved the overall quality of this manuscript.

\end{document}